\documentclass[11pt,a4paper]{article}
\usepackage{caption2}%---change the caption of figures
\usepackage{CJK}
\usepackage{tabls}
\usepackage{bm}
\usepackage{multirow}
\usepackage{amssymb}
\usepackage{amsfonts}
\usepackage{mathrsfs}
\usepackage{empheq}
\usepackage{amsmath}
\usepackage{amsthm}
\usepackage{graphicx}
\usepackage{color}
\usepackage{indentfirst}
\usepackage{hyperref}%pour PDFLaTex
\usepackage{float}
\usepackage{cases}
\usepackage[titletoc]{appendix}
\usepackage{bm}

\usepackage[colorinlistoftodos,english]{todonotes}
\allowdisplaybreaks%

\def\i1n{i=1,\cdots,n}
\def\j1n{j=1,\cdots,n}
\def\ij1n{i,j=1,\cdots,n}

\def\R{\mathbb R}

\def \i{\mathrm i}

 \numberwithin{equation}{section}
\theoremstyle{definition}

 \newtheorem{thm}{\indent Theorem}[section]
 
 \newtheorem{lem}{\indent Lemma}[section]

\theoremstyle{definition}
\theoremstyle{theorem}
\theoremstyle{lemma}

\newcommand{\md}{\mbox{d}}
\newcommand{\be}{\begin{equation}}
\newcommand{\ee}{\end{equation}}
\newcommand{\beq}{\begin{equation*}}
\newcommand{\eeq}{\end{equation*}}

\begin{document}

\title{Global solutions of 2D non-resistive MHD system with constant background magnetic vorticity}

%For each author, make a block with the following macros:

\author{Yuanyuan Qiao\thanks{School of Mathematics Science, Fudan University, Shanghai, P. R. China (20110180038@fudan.edu.cn).}
\and Yi Zhou\thanks{School of Mathematics Science, Fudan University, Shanghai, P. R. China (yizhou@fudan.edu.cn).}}

\date{}

\maketitle

\begin{abstract}
This paper investigates the stabilization effect of a background magnetic vorticity on electrically conducting fluids.
By exploring the dissipation nature of the linearized equations, we prove the global existence of smooth solutions to the two-dimensional incompressible viscous and non-resistive MHD system.
\end{abstract}



\section{Introduction}\label{intro}

The magnetohydrodynamics (MHD) equations are a coupling system of the Navier-Stokes
and Maxwell equations, which can describe the motion of strongly colliding plasma when magnetic diffusion is neglected.
This paper studies the global well-posedness of the following two-dimensional
incompressible viscous and non-resistive MHD system:
\be
\begin{cases}\label{***}
{ \begin{array}{ll} u_t+u\cdot\nabla u-\Delta u+\nabla (P+\frac{1}{2}|B|^2)=B\cdot\nabla B,
\quad (t,x)\in\mathbb{R}^+\times\mathbb{R}^2,\\
B_t+u\cdot\nabla B=B\cdot\nabla u, \\
\nabla\cdot u=\nabla\cdot B=0,\\
t=0: u=u_0(x), B=B_0(x).
 \end{array} }
\end{cases}
\ee
Here $u=(u_1,u_2)^\top$, $B=(B_1,B_2)^\top$ and $P$ represent the velocity field,
magnetic field and scalar pressure of the fluid, respectively.

It is well-known that the viscous and resistive MHD system is globally well-posed in two-dimensional space (see \cite{DL1972,ST1983,AP2008}).
Conversely, for the case of inviscid and non-resistive MHD equations, Bardos, Sulem and Sulem \cite{BSS1988} proved the global well-posedness of the system in the entire $\R^2$ or $\R^3$ when the initial data is close enough to a non-trivial equilibrium state (see \cite{CL2018} for more results).
%C. Cao and J. Wu \cite{CW2011} studied anisotropic partial dissipation
%($\partial_{x_1}^2 u, \partial_{x_2}^2 b$)
%and established the global regularity of classical solutions for any data in $H^2(\mathbb{R}^2)$.
For the case of inviscid and resistive MHD system (only magnetic diffusion),
Lei and Zhou \cite{LZ2009} obtained the global $H^1$ weak solution to the 2D MHD system
(also see \cite{CW2011}), but the global regularity is still a challenging problem.
Recently, Wei and Zhang
\cite{WZ2020} provided the global existence of small smooth solutions in $H^4(\mathbb{T}^2)$.
For more results on this dissipation case, please refer to the literature
\cite{ZZ2018,YY2022,CZZ2022}.

In the case of viscous and non-resistive MHD equations (the case considered in this paper),
the existence of smooth solutions for all time remains a challenging issue even in the two-dimensional whole space for generic smooth initial data.
Lei \cite{L2015} proved the global existence of classical
solutions in three-dimensional space for some axially symmetric initial data
without swirl components for the velocity field and magnetic vorticity field.
Initiated by Lin and Zhang \cite{LZ2014}, numerous efforts have been dedicated to studying the global well-posedness of the MHD system near a constant equilibrium (see \cite{LZ2015,LXZ2015,XZ2015,RWXZ2014,Z2014,HL2014} and the references cited there in).
Recently, using the Lagrangian formulation of the system, Abidi and Zhang \cite{AZ2017} proved the global well-posedness of the 3D MHD system with the initial data close enough to a constant equilibrium state $(u^{(0)}, B^{(0)})=(0, e_3)$. Deng and Zhang \cite{DZ2018} subsequently studied the large time behavior.
Under certain symmetry assumptions, Pan, Zhou and Zhu \cite{PZZ2018} established the global existence of small smooth solutions to the 3D MHD system near the constant equilibrium state on periodic boxes.
It is important to note that they worked in Euclidean coordinates and introduced the time-weighted energy method, which can also be applied to other equations such as the Oldroyd-B model, Euler-Poisson system and compressible MHD system (see \cite{Z2018,ZZ2019,WZ2022} for more details).
Very recently, Chen, Zhang and Zhou \cite{CZZ2022} proved the global well-posedness of the 3D MHD system in the torus $\mathbb{T}^3$ near a constant equilibrium state
satisfying the Diophantine condition.

In this paper we consider a non-constant equilibrium: $B^{(0)}=(x_2,-x_1)^\top$.
Since its vorticity is a non-zero constant, we refer to this equilibrium as a background magnetic vorticity.
We aim to establish the global well-posedness of two-dimensional incompressible viscous and non-resistive MHD system under this background magnetic vorticity.
Taking
$$u^{(0)}=0, \quad P^{(0)}=-|x|^2, \quad B^{(0)}=(x_2,-x_1)^\top,$$
it is easy to see that the steady state $(u^{(0)},P^{(0)},B^{(0)})$ is a special solution to the system (\ref{***}).
Setting $$p=P-P^{(0)}, \quad b=B-B^{(0)},$$ we then deduce the perturbation system of $(u,p,b)$:
\be\label{1.3}
\begin{cases}
{ \begin{array}{ll} u_t+u\cdot\nabla u-\Delta u+\nabla\pi=b\cdot\nabla b-\partial_\theta b,
\quad (t,x)\in\mathbb{R}^+\times\mathbb{R}^2,\\
b_t+u\cdot\nabla b+\nabla\psi=b\cdot\nabla u-\partial_\theta u, \\
\nabla\cdot u=\nabla\cdot b=0,\\
t=0: u=u_0(x), b=b_0(x),\\
 \end{array} }
\end{cases}
\ee
where $\pi=P+\frac{1}{2}|B|^2+\frac{1}{2}|x|^2-\phi$, $\phi$ and $\psi$ are scalar stream functions.
Divergence-free conditions imply the existence of $\phi$ and $\psi$. According to the two-dimensional Biot-Savart law and the notation $\nabla^\bot=(-\partial_2, \partial_1)^\top$, the velocity field $u$ and magnetic field $b$ can be denoted as:
\begin{eqnarray}\label{1.4}
u=\nabla^\bot\psi,\quad b=\nabla^\bot\phi.
\end{eqnarray}
From this, it is clear that $b\cdot\nabla B^{(0)}=\nabla\phi$ and $u\cdot\nabla B^{(0)}=\nabla\psi$, which indicate these two terms can be expressed as gradient forms.
And we can write them together with the pressure term.
%It is worth noting that this approach may not work in the three-dimensional case. When the background magnetic field is $B^{(0)}=(x_2,-x _ 1,0)^\top$,
%the additional linear terms $b\cdot\nabla B^{(0)}$ and $u\cdot\nabla B^{(0)}$ cannot be represented in gradient forms.
In addition, one also verifies that
$B^{(0)}\cdot\nabla=x_2\partial_1-x_1\partial_2:=-\partial_\theta.$
Here and in what follows, we denote $\partial_i=\partial_{x_i}$ (where $i\in{1,2}$) and $\partial_\theta^2 f=\partial_{\theta\theta}f$.
The variables $\theta$ and $r$ (to be used later) represent spatial variables in polar coordinates.

Assuming that the initial data satisfies the following symmetry conditions:
\begin{align}\label{1.1}
\begin{aligned}
&u_{0,1}(x),\quad b_{0,2}(x)\mbox{ are odd functions of}\ x_1,\\
&u_{0,2}(x),\quad b_{0,1}(x)\mbox{ are even functions of}\ x_1,\\
&u_{0,2}(x),\quad b_{0,1}(x)\mbox{ are odd functions of}\ x_2,\\
&u_{0,1}(x),\quad b_{0,2}(x)\mbox{ are even functions of}\ x_2.
\end{aligned}
\end{align}
We then proceed to state our main result:
\begin{thm}\label{thm0.1}
Consider the system (\ref{1.3}) with the initial data satisfying the condition (\ref{1.1}).
If there exists a small constant $\varepsilon>0$, such that
\begin{eqnarray}\label{**}
\|u_0,b_0\|_{\dot{H}^{-\sigma}}+\|u_0,b_0\|_{H^{2s+6}}\leq\varepsilon,
\end{eqnarray}
then the system (\ref{1.3}) has a global smooth solution, provided $s\geq 2$ and $\frac{3}{23}\leq\sigma<1$.
\end{thm}

Our theorem demonstrates that the 2D incompressible non-resistive MHD system
exhibits stability near a non-constant equilibrium: $B^{(0)}=(x_2, -x_1)^\top$.
In the proof of our theorem, it seems difficult to control $\|\nabla u\|_{L_t^1L_x^\infty}$
due to the absence of magnetic diffusion. In order to overcome this difficulty,
we shall make the most of the dissipation nature of linearized equation (\ref{3.1*}). To this end, we first transform the system (\ref{1.3}) into a damped wave-type system. By differentiating with respect to $t$ and making several substitutions, we deduce
\be\label{3.1}
\begin{cases}
{ \begin{array}{ll} u_{tt}-\Delta u_t-\partial_{\theta\theta} u+\nabla q_1=F_t-\partial_{\theta} G, \\
b_{tt}-\Delta b_t-\partial_{\theta\theta} b+\nabla q_2=G_t-\Delta G-\partial_{\theta} F,\\
\nabla\cdot u=\nabla\cdot b=0,
 \end{array} }
\end{cases}
\ee
where
\begin{align*}
\begin{aligned}
q_1:=&\pi_t-\psi_\theta,\\
q_2:=&\psi_t-\partial_\theta\pi-\Delta\psi,\\
F:=&b\cdot\nabla b-u\cdot\nabla u,\\
G:=&b\cdot\nabla u-u\cdot\nabla b.
\end{aligned}
\end{align*}
From the linearized equations:
\be\label{3.1*}
\begin{cases}
{ \begin{array}{ll} u_{tt}-\Delta u_t-\partial_{\theta\theta} u+\nabla q_1=0, \\
b_{tt}-\Delta b_t-\partial_{\theta\theta} b+\nabla q_2=0,\\
\nabla\cdot u=\nabla\cdot b=0,
 \end{array} }
\end{cases}
\ee
it is clear that $u$ and $b$ share the same wave structure.
Moreover, the wave structure exhibits many more regularity properties for $b$ in (\ref{3.1}) compared to the second equation of (\ref{1.3}).
Consequently, we will directly apply the energy estimates for the system (\ref{3.1}).
Additionally, the uniqueness of the solution implies that property (\ref{1.1}) will persist in the time evolution (refer to \cite{PZZ2018, ZZ2018} for detailed information).
As a result, we can conclude that the zeroth Fourier modes for both $u$ and $b$ are zero, i.e.,
\begin{align}\label{1.2}
\int_0^{2\pi}u(t,r,\theta)\md\theta=0, \quad\int_{0}^{2\pi}b(t,r,\theta)\md\theta=0.
\end{align}
This inspires us to utilize the Poincar\'{e} inequality for nonlinear terms, allowing us to acquire the crucial decay estimate: $\int_0^t(1+\tau)^2\|u\|_{\dot{H}^3}^2\md\tau$.
However, this alone is insufficient to control the $L^1$ time integrability of $\|\nabla u\|_{L_x^\infty}\lesssim\|\nabla u\|_{L^2_x}^{\frac{1}{2}}\|\nabla^3 u\|_{L^2_x}^{\frac{1}{2}}$,
since $\|\nabla^3 u\|_{L^2_x}^{\frac{1}{2}}$ fails to provide sufficient decay.
Accordingly, to compensate for this time integrability gap, an
extra energy estimate of order $-\sigma$ ($\frac{3}{23}\leq\sigma<1$) is required.
This provides us with the estimate of $\|u\|_{L_t^2H_x^{1-\sigma}}$.
Armed with these estimates, we can effectively control $\|\nabla u\|_{L_t^1L_x^\infty}$
(see section 3 for comprehensive details, specifically equation (\ref{key})).

The structure of this paper is organized as follows: Second 2 introduces two significant lemmas and energy functionals, which will be extensively used throughout the article.
Section 3 focuses on the estimate of energy functionals stated in Section 2. Finally, Section 4 presents the proof of Theorem \ref{thm0.1}.

We conclude this section by making some notations. Given $a,b\geq 0$, we denote $a\lesssim b$ as $a\leq Cb$, where $C$ denotes a generic constant independent of any particular solutions
or functions. The meaning of $C$ may vary from line to line.
Additionally, we will frequently use $\|f,g\|_{H^k}$ to represent $\|f\|_{k}+\|g\|_{H^k}$, where $\|\cdot\|_{H^k}$ $(k\geq 0)$ is the standard Sobolev norm.
Similarly, we use $\|f,g\|_{H^{-\sigma}}$ to denote $\|f\|_{H^{-\sigma}}+\|g\|_{H^{-\sigma}}$, where $\|\cdot\|_{H^{-\sigma}}$ $(\sigma>0)$ represents the Sobolev norm with a negative index.
The operator $\Lambda$, which signifies the fractional partial derivative, can be defined as
$$\widehat{\Lambda f}(\xi)=|\xi|\hat{f}(\xi),$$
where $\hat{f}(\xi)$ means the standard Fourier transform
$$\hat{f}(\xi)=\frac{1}{2\pi}\int_{\mathbb{R}^2}e^{-ix\cdot\xi}f(x)\md x.$$
Finally, $\langle\cdot,\cdot\rangle$ denotes the $L^2$ inner product.

\section{Preliminaries}

In this part, we make some preparations for the proof of Theorem \ref{thm0.1}.
Firstly, we introduce the following Poincar\'{e} type inequality, which plays a crucial role in the proof of energy estimates.
\begin{lem}\label{lem2.1}
For any given vector $u=(u_1,u_2)^\top\in H^k(\mathbb{R}^2)$ $(k\in \mathbb{N})$ satisfying the following conditions:
\begin{align}\label{001}
\begin{aligned}
&u_1(x)\mbox{ is an odd function of}\ x_1 \mbox{and an even function of}\ x_2,\\
&u_2(x)\mbox{ is an even function of}\ x_1 \mbox{and an odd function of}\ x_2,
\end{aligned}
\end{align}
there holds:
\begin{align}\label{002}
\|u\|_{H^k(\mathbb{R}^2)}\lesssim\|\partial_\theta u\|_{H^k(\mathbb{R}^2)}.
\end{align}
\end{lem}
\begin{proof}
We prove this Lemma in two cases: $k=2m$ and $k=2m+1$, $m\in \mathbb{N}$.

when $k=2m$, according to assumption (\ref{001}) and the standard Poincar\'{e} inequality,
it is easy to see
\begin{align*}
\|u\|_{L^2(\mathbb{R}^2)}\lesssim\|\partial_\theta u\|_{L^2(\mathbb{R}^2)}.
\end{align*}
Note that
\begin{align*}
\|u\|_{\dot{H}^k(\mathbb{R}^2)}=\|\Delta^m u\|_{L^2(\mathbb{R}^2)}
\lesssim\|\partial_\theta\Delta^m u\|_{L^2(\mathbb{R}^2)},
\end{align*}
and the commutant relation $[\Delta, \partial_\theta]=0$, one has
\begin{align*}
\|\partial_\theta\Delta^m u\|_{L^2(\mathbb{R}^2)}
=\|\Delta^m\partial_\theta u\|_{L^2(\mathbb{R}^2)}
=\|\partial_\theta u\|_{\dot{H}^k(\mathbb{R}^2)}.
\end{align*}
Therefore we can deduce that
\begin{align*}
\|u\|_{\dot{H}^k(\mathbb{R}^2)}
\lesssim\|\partial_\theta u\|_{\dot{H}^k(\mathbb{R}^2)}.
\end{align*}

When $k=2m+1$, we first consider the special case $m=0$. Using the fact
$\int_0^{2\pi}\partial_r u\md\theta=0$ (derived directly from $\int_0^{2\pi}u\md\theta=0$) and
the Poincar\'{e} inequality with respect to $\theta$, we get
\begin{align*}
\begin{aligned}
\|u\|_{\dot{H}^1(\mathbb{R}^2)}^2
=&\|\partial_r u\|_{L^2(\mathbb{R}^2)}^2
+\|\frac{1}{r}\partial_\theta u\|_{L^2(\mathbb{R}^2)}^2\\
\lesssim&\|\partial_r\partial_\theta u\|_{L^2(\mathbb{R}^2)}^2
+\|\frac{1}{r}\partial_{\theta\theta}u\|_{L^2(\mathbb{R}^2)}^2\\
=&\|\partial_\theta u\|_{\dot{H}^1(\mathbb{R}^2)}^2,
\end{aligned}
\end{align*}
where $(\partial_r, \frac{1}{r}\partial_\theta)^\top$ represents the gradient in polar coordinates.
By repeating the commutant relation $[\Delta, \partial_\theta]=0$, it is clear that
\begin{align*}
\|u\|_{\dot{H}^k(\mathbb{R}^2)}
=\|\Delta^m\nabla u\|_{L^2(\mathbb{R}^2)}
\lesssim\|\Delta^m\nabla\partial_\theta u\|_{\dot{H}^k(\mathbb{R}^2)}
=\|\partial_\theta u\|_{\dot{H}^k(\mathbb{R}^2)}.
\end{align*}

Based on the above analysis, we can conclude (\ref{002}).
\end{proof}

Actually, the similar estimate is valid for $b$:
\begin{align}\label{007}
\|b\|_{H^k(\mathbb{R}^2)}\lesssim\|\partial_\theta b\|_{H^k(\mathbb{R}^2)}.
\end{align}

Next, let us recall the classical product and commutator estimates.
\begin{lem}\label{lem2.2}\cite{KP1988, KPV1991}
For $k\geq 0$, the following inequalities hold:
\begin{align*}
\|\nabla^k(fg)\|_{L^2}\lesssim&\big(\|f\|_{L^\infty}\|g\|_{H^k}
+\|f\|_{H^k}\|g\|_{L^\infty}\big),\\
\|[\nabla^k,f]g\|_{L^2}\lesssim&\big(\|\nabla f\|_{L^\infty}\|g\|_{H^{k-1}}+\|f\|_{H^{k-1}}\|g\|_{L^\infty}\big),
\end{align*}
where $[\cdot,\cdot]$ stands for the Poisson bracket: $[a,b]=ab-ba$.
\end{lem}

Now we construct the following energy functionals which will enable us to derive the self-contained inequalities. For $s\in\mathbb{N}$ and a small constant $0<\sigma<1$, we take $\sigma=\frac{3}{23}$ and set
\begin{align}
E_0(t)=&\sup_{0\leq\tau\leq t}\big(\|u(\tau)\|_{H^{2s+6}}^2+\|b(\tau)\|_{H^{2s+6}}^2\big)
+\int_0^t\|\nabla u(\tau)\|_{H^{2s+6}}^2\md\tau\label{2.1},\\
E_1(t)=&\sup_{0\leq\tau\leq t}\big(\|u(\tau)\|_{\dot{H}^{-\sigma}}^2+\|b(\tau)\|_{\dot{H}^{-\sigma}}^2\big)
+\int_{0}^t\|u(\tau)\|_{\dot{H}^{1-\sigma}}^2\md\tau\label{2.2},\\
\nonumber e_0(t)=&\sum_{m=0}^s\big\{\sup_{0\leq\tau\leq t}\big(\|u_\tau(\tau),b_\tau(\tau)\|_{\dot{H}^{2m}}^2
+\|\partial_\theta u(\tau),\partial_\theta b(\tau)\|_{\dot{H}^{2m}}^2
+\|u(\tau),b(\tau)\|_{\dot{H}^{2m+2}}^2\big)\\
&+\int_0^t\big(\|u_\tau(\tau),b_\tau(\tau)\|_{\dot{H}^{2m+1}}^2
+\|\partial_\theta u(\tau),\partial_\theta b(\tau)\|_{\dot{H}^{2m+1}}^2\big)\md\tau\big\}\label{2.3},\\
\nonumber e_1(t)=&\sum_{m=1}^{s-1}\big\{\sup_{0\leq\tau\leq t}(1+\tau)^2\big(\|u_\tau(\tau),b_\tau(\tau)\|_{\dot{H}^{2m}}^2
+\|\partial_\theta u(\tau),\partial_\theta b(\tau)\|_{\dot{H}^{2m}}^2
+\|u(\tau),b(\tau)\|_{\dot{H}^{2m+2}}^2\big)\\
&+\int_0^t(1+\tau)^2\big(\|u_\tau(\tau),b_\tau(\tau)\|_{\dot{H}^{2m+1}}^2
+\|\partial_\theta u(\tau),\partial_\theta b(\tau)\|_{\dot{H}^{2m+1}}^2\big)\md\tau\big\}\label{2.4}.
\end{align}

\section{Energy estimates}

The purpose of this section is to derive the uniform boundness of (\ref{2.1})-(\ref{2.4}).
We start by estimating $E_0(t)$.
\begin{lem}\label{lem3.1}
From the definition in (\ref{2.1}), for $s\geq 2$, we have
\begin{align}\label{1}
E_0(t)\lesssim\varepsilon^2+E_0(t)^{\frac{3}{2}}+E_0(t)e_0(t)^{\frac{1}{2}}
+E_0(t)E_1(t)^{\frac{1}{2(2+\sigma)}}e_1(t)^{\frac{1+\sigma}{2(2+\sigma)}}.
\end{align}
\end{lem}
\begin{proof}
First of all, one can easily derive the following basic energy law
\begin{eqnarray*}
\|u\|_{L^2}^2+\|b\|_{L^2}^2+2\int_0^t\|\nabla u\|_{L^2}^2\md\tau=\|u_0\|_{L^2}^2+\|b_0\|_{L^2}^2.
\end{eqnarray*}
Next, we deal with the highest order energy. Operating the $\Delta^{2s+6}$ derivative on the equations of (\ref{1.3}) and taking the inner product with $u$ and $b$, respectively, we obtain
\begin{align}\label{*}
\begin{aligned}
&\frac{1}{2}\frac{\md}{\md t}\big(\|u\|_{\dot{H}^{2s+6}}^2+\|b\|_{\dot{H}^{2s+6}}^2\big)
+\|\nabla u\|_{\dot{H}^{2s+6}}^2\\
=&-\langle\nabla^{2s+6}u,[\nabla^{2s+6}, u\cdot\nabla]u\rangle
+\langle\nabla^{2s+6}u,[\nabla^{2s+6}, b\cdot\nabla]b\rangle\\
&-\langle\nabla^{2s+6}b,[\nabla^{2s+6}, u\cdot\nabla]b\rangle
+\langle\nabla^{2s+6}b,[\nabla^{2s+6}, b\cdot\nabla]u\rangle\\
:=&\sum_{i=1}^4I_i,
\end{aligned}
\end{align}
where we have used the divergence-free conditions and the fact that
\begin{align*}
\langle\Delta^{2s+6}\partial_\theta b,u\rangle
+\langle\Delta^{2s+6}\partial_\theta u,b\rangle=0.
\end{align*}
By H\"{o}lder's inequality, the commutator estimate and Sobolev imbedding theorem, we deduce
\begin{align*}
\int_0^t|I_1|\md\tau
\lesssim&\int_0^t\|u\|_{\dot{H}^{2s+6}}^2\|\nabla u\|_{L^{\infty}}\md\tau\\
\lesssim&\int_0^t\|u\|_{\dot{H}^{2s+6}}^2\|\nabla u\|_{H^2}\md\tau\\
%\leq& \sup_{0\leq\tau\leq t}\|\nabla u\|_{H^2}
%\int_0^t\|u\|_{\dot{H}^{2s+6}}^2\md\tau\\
\lesssim&E_0(t)^{\frac{3}{2}}.
\end{align*}
Similarly, one has
\begin{align*}
|I_2|
\lesssim& \|u\|_{\dot{H}^{2s+6}}\|b\|_{\dot{H}^{2s+6}}\|\nabla b\|_{H^2}.
\end{align*}
Poincar\'{e} inequality (\ref{007}) further leads to
\begin{align*}
\int_0^t|I_2|\md\tau
\lesssim& \int_0^t\|u\|_{\dot{H}^{2s+6}}\|b\|_{\dot{H}^{2s+6}}\|\nabla \partial_\theta b\|_{H^2}\md\tau\\
\lesssim&\sup_{0\leq\tau\leq t}\|b\|_{\dot{H}^{2s+6}}
\big(\int_0^t\|u\|_{\dot{H}^{2s+6}}^2\md\tau\big)^{\frac{1}{2}}
\big(\int_0^t\|\nabla \partial_\theta b\|_{H^2}^2\md\tau\big)^{\frac{1}{2}}\\
\lesssim&E_0(t)e_0(t)^{\frac{1}{2}}.
\end{align*}
For $I_3$ and $I_4$, again using H\"{o}lder's inequality, the commutator estimate and  Sobolev imbedding theorem, one has
\begin{align*}
|I_3|+|I_4|
\lesssim& \|u\|_{\dot{H}^{2s+6}}\|b\|_{\dot{H}^{2s+6}}\|\nabla b\|_{H^2}
+\|b\|_{\dot{H}^{2s+6}}^2\|\nabla u\|_{L^\infty}.
\end{align*}
%By H\"{o}lder's inequality once again, along with (\ref{007}), we can bound the first term
%\begin{align*}
%&\int_0^t\|u\|_{\dot{H}^{2s+6}}\|b\|_{\dot{H}^{2s+6}}\|\nabla b\|_{L^{\infty}}\md\tau\\
%\lesssim& \sup_{0\leq\tau\leq t}\|b\|_{\dot{H}^{2s+6}}
%\big(\int_0^t\|u\|_{\dot{H}^{2s+6}}^2\md\tau\big)^{\frac{1}{2}}
%\big(\int_0^t\|\nabla \partial_\theta b\|_{H^2}^2\md\tau\big)^{\frac{1}{2}}\\
%\lesssim&E_0(t)e_0(t)^{\frac{1}{2}}.
%\end{align*}
The first term on the right-hand side of the above inequality can be treated in the same way as $I_2$.
However, due to the non-resistive property of our system, the second term turns out to be more complicated.
Although the symmetry condition (\ref{1.1}) provide us some decay estimates, more is needed to control this term. To overcome this difficulty, we introduce the extra energy functional of $-\sigma$ order, denoted as $E_1(t)$, which implies that $u\in L_t^2H_x^{1-\sigma}$ ($\frac{3}{23}\leq\sigma<1$).
Then, by utilizing Gagliardo-Nirenberg's inequality and Lemma \ref{lem2.1}, we can estimate the second term as follows
\begin{align}\label{key}
\begin{aligned}
&\int_0^t\|b\|_{\dot{H}^{2s+6}}^2\|\nabla u\|_{L^{\infty}}\md\tau\\
%\lesssim&
%\int_0^t\|b\|_{\dot{H}^{2s+6}}^2\|\nabla u\|_{L^2}^{\frac{1}{2}}\|\nabla^3 u\|_{L^2}^{\frac{1}{2}}\md\tau\\
\lesssim&
\int_0^t\|b\|_{\dot{H}^{2s+6}}^2\|u\|_{\dot{H}^{1-\sigma}}^{\frac{1}{2+\sigma}}
\|\partial_\theta u\|_{\dot{H}^3}^{\frac{\sigma}{2(2+\sigma)}+\frac{1}{2}}\md\tau\\
=&C\int_0^t(1+\tau)^{-\frac{1}{2}-\frac{\sigma}{2(2+\sigma)}}
\|b\|_{\dot{H}^{2s+6}}^2\|u\|_{\dot{H}^{1-\sigma}}^{\frac{1}{2+\sigma}}
(1+\tau)^{\frac{\sigma}{2(2+\sigma)}+\frac{1}{2}}\md\tau
\|\partial_\theta u\|_{\dot{H}^3}^{\frac{\sigma}{2(2+\sigma)}+\frac{1}{2}}\md\tau\\
\lesssim&\sup_{0\leq\tau\leq t}\|b\|_{H^{2s+6}}^2
\big(\int_0^t\|u\|_{\dot{H}^{1-\sigma}}^2\md\tau\big)^{\frac{1}{2(2+\sigma)}}
\big(\int_0^t(1+\tau)^2
\|\partial_\theta u\|_{\dot{H}^3}^2\md\tau\big)^{\frac{1+\sigma}{2(2+\sigma)}}\\
\lesssim&E_0(t)E_1(t)^{\frac{1}{2(2+\sigma)}}e_1(t)^{\frac{1+\sigma}{2(2+\sigma)}}.
\end{aligned}
\end{align}

Integrating (\ref{*}) with respect to time and inserting the above estimates leads to (\ref{1}).
\end{proof}

In the following lemma, we aim to obtain the estimate of $\|u\|_{L_t^2H_x^{1-\sigma}}$.
\begin{lem}\label{lem3.2}
From the definition in (\ref{2.2}), for $0<\sigma< 1$, we have
\begin{align}\label{2}
E_1(t)\lesssim\varepsilon^2+E_1(t)e_0(t)^{\frac{1}{2}}.
\end{align}
\end{lem}
\begin{proof}
Applying $\Lambda^{-\sigma}$ to the MHD equations (\ref{1.3}), and taking $L^2$ inner product with ($\Lambda^{-\sigma}u, \Lambda^{-\sigma}b$), we have
\begin{align*}
\begin{aligned}
&\frac{1}{2}\frac{\md}{\md t}\big(\|u\|_{\dot{H}^{-\sigma}}^2+\|b\|_{\dot{H}^{-\sigma}}^2\big)
+\|\nabla u\|_{\dot{H}^{-\sigma}}^2\\
=&\langle\Lambda^{-\sigma}u,\Lambda^{-\sigma}(b\cdot\nabla b-u\cdot\nabla u)\rangle
+\langle\Lambda^{-\sigma}b,\Lambda^{-\sigma}(b\cdot\nabla u-u\cdot\nabla b)\rangle\\
:=&M_1+M_2,
\end{aligned}
\end{align*}
where we have used the fact that
$$\langle\Lambda^{-\sigma}u,\Lambda^{-\sigma}\partial_\theta b\rangle
+\langle\Lambda^{-\sigma}b,\Lambda^{-\sigma}\partial_\theta u\rangle=0.$$
Using divergence-free conditions, integrating by parts, H\"{o}lder's inequality and  Sobolev imbedding theorem,
for $\frac{1}{q}=\frac{1}{2}+\frac{\sigma}{2}=\frac{1}{2}+\frac{1}{p}$, we have
\begin{align*}
\begin{aligned}
M_1
=&\langle\Lambda^{1-\sigma}u,\Lambda^{-\sigma}(u\otimes u-b\otimes b)\rangle\\
\leq&\|u\|_{\dot{H}^{1-\sigma}}\|u\otimes u-b\otimes b\|_{\dot{H}^{-\sigma}}\\
\lesssim&\|u\|_{\dot{H}^{1-\sigma}}\|u\otimes u-b\otimes b\|_{L^q}\\
\lesssim&\|u\|_{\dot{H}^{1-\sigma}}\big(\|u\|_{L^p}\|u\|_{L^2}+\|b\|_{L^p}\|b\|_{L^2}\big).\\
\end{aligned}
\end{align*}
By Gagliardo-Nirenberg's inequalities:
\begin{align*}
\|u\|_{L^p}\lesssim&\|u\|_{\dot{H}^{-\sigma}}^{\frac{2}{p(1+\sigma)}}
\|u\|_{\dot{H}^1}^{1-\frac{2}{p(1+\sigma)}},\\
\|u\|_{L^2}\lesssim&\|u\|_{\dot{H}^{-\sigma}}^{\frac{1}{1+\sigma}}
\|u\|_{\dot{H}^1}^{\frac{\sigma}{1+\sigma}},
\end{align*}
Lemma \ref{lem2.1} and (\ref{007}), we further deduce
\begin{align*}
\int_0^tM_1\md\tau
\lesssim&\int_0^t\|u\|_{\dot{H}^{1-\sigma}}\big(\|u\|_{\dot{H}^{-\sigma}}\|u\|_{\dot{H}^1}
+\|b\|_{\dot{H}^{-\sigma}}\|b\|_{\dot{H}^1}\big)\md\tau\\
\lesssim&\int_0^t\|u\|_{\dot{H}^{1-\sigma}}\big(\|u\|_{\dot{H}^{-\sigma}}\|\partial_\theta u\|_{\dot{H}^1}
+\|b\|_{\dot{H}^{-\sigma}}\|\partial_\theta b\|_{\dot{H}^1}\big)\md\tau\\
\lesssim&\sup_{0\leq\tau\leq t}\|u\|_{\dot{H}^{-\sigma}}\big(\int_0^t\|u\|_{\dot{H}^{1-\sigma}}^2\md \tau\big)^{\frac{1}{2}}\big(\int_0^t\|\partial_\theta u\|_{\dot{H}^1}^2\md\tau\big)^{\frac{1}{2}}\\
&+\sup_{0\leq\tau\leq t}\|b\|_{\dot{H}^{-\sigma}}\big(\int_0^t\|u\|_{\dot{H}^{1-\sigma}}^2\md \tau\big)^{\frac{1}{2}}\big(\int_0^t\|\partial_\theta b\|_{\dot{H}^1}^2\md\tau\big)^{\frac{1}{2}}\\
\lesssim& E_1(t)e_0(t)^{\frac{1}{2}}.
\end{align*}
In the same manner, for $\frac{1}{q}=\frac{1}{2}+\frac{\sigma}{2}=\frac{1}{2}+\frac{1}{p}$, we have
\begin{align*}
\begin{aligned}
M_2
=&\langle\Lambda^{1-\sigma}b,\Lambda^{-\sigma}(u\otimes b-b\otimes u)\rangle\\
\leq&\|b\|_{\dot{H}^{1-\sigma}}\|u\otimes b-b\otimes u\|_{\dot{H}^{-\sigma}}\\
\lesssim&\|b\|_{\dot{H}^{1-\sigma}}\|u\otimes b-b\otimes u\|_{L^q}\\
\lesssim&\|b\|_{\dot{H}^{1-\sigma}}\|b\|_{L^2}\|u\|_{L^p}.\\
\end{aligned}
\end{align*}
By Sobolev imbedding inequality $\|u\|_{L^p}\lesssim\|u\|_{\dot{H}^{1-\sigma}}$
and the interpolation inequalities:
\begin{align*}
\begin{aligned}
\|b\|_{\dot{H}^{1-\sigma}}\lesssim\|b\|_{\dot{H}^{-\sigma}}^{\frac{\sigma}{\sigma+1}}
\|b\|_{\dot{H}^1}^{1-\frac{\sigma}{\sigma+1}},\\
\|b\|_{L^2}\lesssim\|b\|_{\dot{H}^{-\sigma}}^{\frac{1}{\sigma+1}}
\|b\|_{\dot{H}^1}^{1-\frac{1}{\sigma+1}},
\end{aligned}
\end{align*}
as well as (\ref{007}), we further deduce
\begin{align*}
M_2\lesssim
\|b\|_{\dot{H}^{-\sigma}}\|\partial_\theta b\|_{\dot{H}^1}\|u\|_{\dot{H}^{1-\sigma}}.
\end{align*}
Thus
\begin{align*}
\begin{aligned}
\int_0^tM_2\md\tau
\lesssim&\sup_{0\leq\tau\leq t}\|b\|_{\dot{H}^{-\sigma}}
\big(\int_0^t\|u\|_{\dot{H}^{1-\sigma}}^2\md\tau\big)^{\frac{1}{2}}
\big(\int_0^t\|\partial_\theta b\|_{\dot{H}^1}^2\md\tau\big)^{\frac{1}{2}}\\
\lesssim&E_1(t)e_0(t)^{\frac{1}{2}}.
\end{aligned}
\end{align*}

Collecting the above estimates, we get the estimate of $E_1(t)$.
\end{proof}

The proof of the decay estimate $e_1(t)$ relies crucially on the following lemma.
Taking advantage of the damped wave-type equations (\ref{3.1}), we shall establish the lower-order energy estimates.
\begin{lem}\label{lem3.3}
From the definition in (\ref{2.3}), for $s\geq 1$, we have
\begin{align}\label{3}
e_0(t)\lesssim\varepsilon^2+e_0(t)^{\frac{3}{2}}
+E_0(t)^{\frac{1}{2}}e_0(t)
+E_0(t)e_0(t)^{\frac{1}{2}}
+E_0(t)^{\frac{3}{4}}e_0(t)^{\frac{3}{4}}.
\end{align}
\end{lem}
\begin{proof}
Recall the damped wave-type system (\ref{3.1}):
\begin{align*}
\begin{cases}
{ \begin{array}{ll} u_{tt}-\Delta u_t-\partial_{\theta\theta} u+\nabla q_1=F_t-\partial_{\theta} G, \\
b_{tt}-\Delta b_t-\partial_{\theta\theta} b+\nabla q_2=G_t-\Delta G-\partial_{\theta} F,
 \end{array} }
\end{cases}
\end{align*}
where
$F=b\cdot\nabla b-u\cdot\nabla u$,
$G=b\cdot\nabla u-u\cdot\nabla b$
are nonlinear terms.

To get the positive energy estimates including the time integrability of $\partial_\theta u$ and $\partial_\theta b$, we need to make two estimates, multiply the results by appropriate numbers and add them up.
Firstly, taking $L^2$ inner product with $(\Delta^{2m} u_t,\Delta^{2m} b_t)$ for the above equations, one has
\begin{align}\label{3.2}
\begin{aligned}
\frac{1}{2}\frac{\md}{\md t}\big(\|u_t,b_t\|_{\dot{H}^{2m}}^2+\|\partial_\theta u,\partial_\theta b\|_{\dot{H}^{2m}}^2\big)
+\|u_t, b_t\|_{\dot{H}^{2m+1}}^2
=N_1.
\end{aligned}
\end{align}
Secondly, taking $L^2$ inner product with $(-\Delta^{2m+1} u,-\Delta^{2m+1} b)$ for (\ref{3.1}), we obtain
\begin{align}\label{3.3}
\begin{aligned}
\frac{\md}{\md t}\big(-\langle\Delta^{2m+1}u, u_t\rangle
-\langle\Delta^{2m+1}b, b_t\rangle
+\frac{1}{2}\|u,b\|_{\dot{H}^{2m+2}}^2\big)\\
-\|u_t,b_t\|_{\dot{H}^{2m+1}}^2+\|\partial_\theta u,\partial_\theta b\|_{\dot{H}^{2m+1}}^2
=N_2.
\end{aligned}
\end{align}
Here $N_i$ $(i=1,2)$ are nonlinear terms:
\begin{align*}
\begin{aligned}
N_1=&\langle\Delta^{2m}u_t, F_t-\partial_{\theta} G\rangle+\langle\Delta^{2m}b_t, G_t-\Delta G-\partial_{\theta} F\rangle,\\
N_2=&-\langle\Delta^{2m+1}u, F_t-\partial_{\theta} G\rangle-\langle\Delta^{2m+1}b, G_t-\Delta G-\partial_{\theta} F\rangle.\\
\end{aligned}
\end{align*}
Thirdly, multiplying (\ref{3.2}) by a suitable number and adding (\ref{3.3}),
and noting the following inequalities:
\begin{align*}
\big|\langle\Delta^{2m+1}u, u_t\rangle\big|
\lesssim& \|u\|_{\dot{H}^{2m+2}}\|u_t\|_{\dot{H}^{2m}},\\
\big|\langle\Delta^{2m+1}b, b_t\rangle\big|
\lesssim& \|b\|_{\dot{H}^{2m+2}}\|b_t\|_{\dot{H}^{2m}}.
\end{align*}
Further from Young's inequality, we can deduce that
\begin{align}\label{3.5}
\begin{aligned}
\frac{\md}{\md t}
\big(\|u_t,b_t\|_{\dot{H}^{2m}}^2+\|\partial_\theta u,\partial_\theta b\|_{\dot{H}^{2m}}^2
%-\langle\Delta^{m+1}u, \Delta^mu_t\rangle
%-\langle\Delta^{m+1}b, \Delta^mb_t\rangle
+\|u,b\|_{\dot{H}^{2m+2}}^2\big)\\
+\|u_t, b_t\|_{\dot{H}^{2m+1}}^2
+\|\partial_\theta u,\partial_\theta b\|_{\dot{H}^{2m+1}}^2
\lesssim N_1+N_2.
\end{aligned}
\end{align}
Next, we deal with the nonlinear term $N_1$. Since $u_t$, $b_t$, $\partial_\theta u$ and $\partial_\theta b$ have the same regularity on the left-hand side of (\ref{3.5}), for brevity, we will put the same order derivative terms together and estimate only one of them. Let us start with the lower-order derivative terms:
\begin{align}\label{3.7}
\begin{aligned}
&\langle\Delta^{2m}u_t, F_t-\partial_{\theta} G\rangle+\langle\Delta^{2m}b_t, G_t-\partial_{\theta} F\rangle\\
=&\langle\Delta^mu_t, \Delta^m(b\cdot\nabla b)_t\rangle
-\langle\Delta^mu_t,\Delta^m(u\cdot\nabla u)_t\rangle\\
&-\langle\Delta^mu_t, \Delta^m\partial_\theta(b\cdot\nabla u)\rangle
+\langle\Delta^mu_t,\Delta^m\partial_\theta(u\cdot\nabla b)\rangle\\
&+\langle\Delta^mb_t, \Delta^m(b\cdot\nabla u)_t\rangle
-\langle\Delta^mb_t,\Delta^m(u\cdot\nabla b)_t\rangle\\
&-\langle\Delta^mb_t, \Delta^m\partial_\theta(b\cdot\nabla b)\rangle
+\langle\Delta^mb_t,\Delta^m\partial_\theta(u\cdot\nabla u)\rangle.
\end{aligned}
\end{align}
We only estimate the first term on the right-hand side of (\ref{3.7}). The remaining terms can be treated in a similar way with some modifications and we will omit the details. To overcome the lack of the space-time $L^2$ norm for $u_t$, $b_t$, $\partial_\theta u$ and $\partial_\theta b$, we handle it in two cases: $m\geq 1$ and $m=0$.

For $m\geq 1$,
by the divergence-free condition, Lemma \ref{lem2.2}, Sobolev imbedding theorem and (\ref{007}), we have
\begin{align*}
\begin{aligned}
&\langle\Delta^mu_t, \Delta^m(b\cdot\nabla b)_t\rangle\\
=&\langle\nabla^{2m}u_t, \nabla^{2m+1}(b\otimes b)_t\rangle\\
\lesssim&\|u_t\|_{\dot{H}^{2m}}\big(\|b\|_{\dot{H}^{2m+1}}\|b_t\|_{L^\infty}
+\|b\|_{L^\infty}\|b_t\|_{\dot{H}^{2m+1}}\big)\\
\lesssim&\|u_t\|_{\dot{H}^{2m}}\big(\|\partial_\theta b\|_{\dot{H}^{2m+1}}\|b_t\|_{H^2}
+\|\partial_\theta b\|_{H^2}\|b_t\|_{\dot{H}^{2m+1}}\big).
\end{aligned}
\end{align*}
Hence
\begin{align*}
\begin{aligned}
&\int_0^t\langle\Delta^mu_\tau, \Delta^m(b\cdot\nabla b)_\tau\rangle\md\tau\\
\leq&\sup_{0\leq\tau\leq t}\|b_\tau\|_{H^2}
\big(\int_0^t\|u_\tau\|_{\dot{H}^{2m}}^2\md\tau\big)^{\frac{1}{2}}
\big(\int_0^t\|\partial_\theta b\|_{\dot{H}^{2m+1}}\md\tau\big)^{\frac{1}{2}}\\
&+\sup_{0\leq\tau\leq t}\|\partial_\theta b\|_{H^2}
\big(\int_0^t\|u_\tau\|_{\dot{H}^{2m}}^2\md\tau\big)^{\frac{1}{2}}
\big(\int_0^t\|b_\tau\|_{\dot{H}^{2m+1}}\md\tau\big)^{\frac{1}{2}}\\
\lesssim& e_0(t)^{\frac{3}{2}},
\end{aligned}
\end{align*}
where we have used the interpolation inequality:
$$\|u_t\|_{\dot{H}^{2m}}
\lesssim\|u_t\|_{\dot{H}^{2m-1}}^{\frac{1}{2}}\|u_t\|_{\dot{H}^{2m+1}}^{\frac{1}{2}}.$$

For $m=0$, using H\"{o}lder's inequality, Sobolev imbedding theorem and (\ref{007}), we obtain
\begin{align*}
\begin{aligned}
&\langle u_t, (b\cdot\nabla b)_t\rangle\\
=&\langle u_t, b_t\cdot\nabla b+b\cdot\nabla b_t\rangle\\
\lesssim&\|u_t\|_{L^4}\|b_t\|_{L^4}\|\nabla b\|_{L^2}
+\|u\|_{L^4}\|b_t\|_{L^4}\|\nabla b_t\|_{L^2}\\
\lesssim&\|u_t\|_{L^2}^{\frac{1}{2}}\|u_t\|_{\dot{H}^1}^{\frac{1}{2}}
\|b_t\|_{L^2}^{\frac{1}{2}}\|b_t\|_{\dot{H}^1}^{\frac{1}{2}}\|\partial_\theta b\|_{\dot{H}^1}
+\|u_t\|_{L^2}^{\frac{1}{2}}\|u_t\|_{\dot{H}^1}^{\frac{1}{2}}
\|\partial_\theta b\|_{L^2}^{\frac{1}{2}}
\|\partial_\theta b\|_{\dot{H}^1}^{\frac{1}{2}}\|b_t\|_{\dot{H}^1}.
\end{aligned}
\end{align*}
Therefore
\begin{align*}
\begin{aligned}
&\int_0^t\langle u_\tau, (b\cdot\nabla b)_\tau\rangle\md\tau\\
\lesssim&\sup_{0\leq\tau\leq t}\|u_\tau\|_{L^2}^{\frac{1}{2}}\|b_\tau\|_{L^2}^{\frac{1}{2}}
\big(\int_0^t\|u_\tau\|_{\dot{H}^1}^2\md\tau\big)^{\frac{1}{4}}
\big(\int_0^t\|b_\tau\|_{\dot{H}^1}^2\md\tau\big)^{\frac{1}{4}}
\big(\int_0^t\|\partial_\theta b\|_{\dot{H}^1}^2\md\tau\big)^{\frac{1}{2}}\\
&+\sup_{0\leq\tau\leq t}\|u_\tau\|_{L^2}^{\frac{1}{2}}\|\partial_\theta b\|_{L^2}^{\frac{1}{2}}
\big(\int_0^t\|u_\tau\|_{\dot{H}^1}^2\md\tau\big)^{\frac{1}{4}}
\big(\int_0^t\|\partial_\theta b\|_{\dot{H}^1}^2\md\tau\big)^{\frac{1}{4}}
\big(\int_0^t\|b_\tau\|_{\dot{H}^1}^2\md\tau\big)^{\frac{1}{2}}\\
\lesssim& e_0(t)^{\frac{3}{2}}.
\end{aligned}
\end{align*}

We now handle the remaining higher-order term of $N_1$. From the incompressible constraints, Lemma \ref{lem2.2}, interpolation inequalities, Sobolev imbedding theorem and (\ref{007}), we deduce that
\begin{align*}
\begin{aligned}
&\langle\Delta^{2m}b_t, -\Delta G\rangle\\
=&\langle\nabla^{2m+1}b_t, \nabla^{2m+2}(b\otimes u-u\otimes b)\rangle\\
\lesssim&\|b_t\|_{\dot{H}^{2m+1}}\big(\|u\|_{L^\infty}\|b\|_{\dot{H}^{2m+2}}
+\|b\|_{L^\infty}\|u\|_{\dot{H}^{2m+2}}\big)\\
\lesssim&\|b_t\|_{\dot{H}^{2m+1}}\big(\|u\|_{L^2}^{\frac{2}{3}}\|\partial_\theta u\|_{\dot{H}^3}^{\frac{1}{3}}
\|\partial_\theta b\|_{\dot{H}^{2m+1}}^{\frac{1}{2}}\|b\|_{\dot{H}^{2m+3}}^{\frac{1}{2}}
+\|b\|_{H^2}\|u\|_{\dot{H}^{2m+2}}\big).
\end{aligned}
\end{align*}
Therefore
\begin{align*}
\begin{aligned}
&\int_0^t\langle\Delta^{2m}b_\tau, -\Delta G\rangle\md\tau\\
\lesssim&\sup_{0\leq\tau\leq t}\|u\|_{L^2}^{\frac{1}{2}}
\|b\|_{H^{2s+3}}^{\frac{1}{2}}
\big(\int_0^t\|\nabla u\|_{H^1}^2\md\tau\big)^{\frac{1}{4}}
\big(\int_0^t\|b_\tau\|_{\dot{H}^{2m+1}}^2\md\tau\big)^{\frac{1}{2}}
\big(\int_0^t\|\partial_\theta b\|_{\dot{H}^{2m+1}}^2\md\tau\big)^{\frac{1}{4}}\\
&+\sup_{0\leq\tau\leq t}\|b\|_{H^2}
\big(\int_0^t\|\nabla u\|_{H^{2s+1}}^2\md\tau\big)^{\frac{1}{2}}
\big(\int_0^t\|b_\tau\|_{\dot{H}^{2m+1}}^2\md\tau\big)^{\frac{1}{2}}\\
\lesssim& E_0(t)^{\frac{3}{4}}e_0(t)^{\frac{3}{4}}+E_0(t)e_0(t)^{\frac{1}{2}}.
\end{aligned}
\end{align*}

Like $N_1$, we group the terms of $N_2$ based on the order of the spatial derivative. The same lower-order derivative terms are as follows:
\begin{align}\label{003}
\begin{aligned}
&\langle\Delta^{2m+1}u, F_t-\partial_{\theta} G\rangle+\langle\Delta^{2m+1}b, G_t-\partial_{\theta} F\rangle\\
=&-\langle\nabla^{2m+1}u, \nabla^{2m+1}(b\cdot\nabla b)_t\rangle
+\langle\nabla^{2m+1}u,\nabla^{2m+1}(u\cdot\nabla u)_t\rangle\\
&+\langle\nabla^{2m+1}u, \nabla^{2m+1}\partial_\theta(b\cdot\nabla u)\rangle
-\langle\nabla^{2m+1}u,\nabla^{2m+1}\partial_\theta(u\cdot\nabla b)\rangle\\
&-\langle\nabla^{2m+1}b, \nabla^{2m+1}(b\cdot\nabla u)_t\rangle
+\langle\nabla^{2m+1}b,\nabla^{2m+1}(u\cdot\nabla b)_t\rangle\\
&+\langle\nabla^{2m+1}b, \nabla^{2m+1}\partial_\theta(b\cdot\nabla b)\rangle
-\langle\nabla^{2m+1}b,\nabla^{2m+1}\partial_\theta(u\cdot\nabla u)\rangle.
\end{aligned}
\end{align}
It can be seen that the maximum order of derivatives for $u_t$, $b_t$, $\partial_\theta u$ and $\partial_\theta b$ in above is $2m+2$, which exceeds their regularity on the left-hand side of (\ref{3.5}).
Our strategy is to decrease the derivative of $u_t$, $b_t$, $\partial_\theta u$ and $\partial_\theta b$ from the order of $2m+2$ to $2m+1$ through integration by parts.
On the other hand, using interpolation inequalities,
we can control the $2m+2$ derivative of $u, b$ via the higher-order energy $E_0(t)$.
Next, we only estimate the first term of (\ref{003}). The remaining terms can be estimated in a similar manner with appropriate adjustments. Further details will be omitted.
By utilizing $\nabla\cdot b=0$, Lemma \ref{lem2.2}, Sobolev imbedding theorem and (\ref{007}), we obtain
\begin{align*}
\begin{aligned}
&-\langle\nabla^{2m+1}u, \nabla^{2m+1}(b\cdot\nabla b)_t\rangle\\
=&\langle\nabla^{2m+2}u, \nabla^{2m+1}(b\otimes b)_t\rangle\\
\lesssim&\|u\|_{\dot{H}^{2m+2}}\big(\|b\|_{\dot{H}^{2m+1}}\|b_t\|_{L^\infty}
+\|b\|_{L^\infty}\|b_t\|_{\dot{H}^{2m+1}}\big)\\
\lesssim&\|u\|_{\dot{H}^{2m+2}}\big(\|\partial_\theta b\|_{\dot{H}^{2m+1}}\|b_t\|_{H^2}
+\|\partial_\theta b\|_{H^2}\|b_t\|_{\dot{H}^{2m+1}}\big).
\end{aligned}
\end{align*}
Hence
\begin{align*}
\begin{aligned}
&\int_0^t-\langle\nabla^{2m+1}u, \nabla^{2m+1}(b\cdot\nabla b)_\tau\rangle\md\tau\\
\lesssim&\sup_{0\leq\tau\leq t}
\|b_\tau\|_{H^2}
\big(\int_0^t\|\nabla u\|_{H^{2s+1}}^2\md\tau\big)^{\frac{1}{2}}
\big(\int_0^t\|\partial_\theta b\|_{\dot{H}^{2m+1}}^2\md\tau\big)^{\frac{1}{2}}\\
&+\sup_{0\leq\tau\leq t}\|\partial_\theta b\|_{H^2}
\big(\int_0^t\|\nabla u\|_{H^{2s+1}}^2\md\tau\big)^{\frac{1}{2}}
\big(\int_0^t\|b_\tau\|_{\dot{H}^{2m+1}}^2\md\tau\big)^{\frac{1}{2}}\\
\lesssim& E_0(t)^{\frac{1}{2}}e_0(t).
\end{aligned}
\end{align*}

For the remaining term of $N_2$, using Lemma \ref{lem2.2}, interpolation inequalities, Sobolev imbedding theorem and (\ref{007}), we have
\begin{align*}
&\langle\Delta^{2m+1}b, \Delta G\rangle\\
=&\langle\nabla^{2m+2}b, \nabla^{2m+2}(b\cdot\nabla u-u\cdot\nabla b)\rangle\\
=&\langle\nabla^{2m+2}b, \nabla^{2m+2}(b\cdot\nabla u)\rangle
-\langle\nabla^{2m+2}b, [\nabla^{2m+2},u\cdot\nabla]b\rangle\\
\lesssim&\|b\|_{\dot{H}^{2m+2}}\big(\|b\|_{\dot{H}^{2m+2}}\|\nabla u\|_{L^\infty}
+\|b\|_{L^\infty}\|\nabla u\|_{\dot{H}^{2m+2}}
+\|u\|_{\dot{H}^{2m+2}}\|\nabla b\|_{L^\infty}\big)\\
\lesssim&\|\partial_\theta b\|_{\dot{H}^{2m+1}}\|b\|_{\dot{H}^{2m+3}}\|\nabla u\|_{H^2}
+\|\partial_\theta b\|_{\dot{H}^{2m+1}}^{\frac{2}{3}}\|b\|_{\dot{H}^{2m+4}}^{\frac{1}{3}}
\|b\|_{L^2}^{\frac{2}{3}}\|\partial_\theta b\|_{\dot{H}^3}^{\frac{1}{3}}
\|\nabla u\|_{\dot{H}^{2m+2}}\\
&+\|b\|_{\dot{H}^{2m+2}}
\|u\|_{\dot{H}^{2m+2}}
\|\nabla \partial_\theta b\|_{H^2}.
\end{align*}
Therefore
\begin{align*}
&\int_0^t\langle\Delta^{2m+1}b, \Delta G\rangle\md\tau\\
\lesssim&\sup_{0\leq\tau\leq t}\|b\|_{H^{2s+3}}
\big(\int_0^t\|\partial_\theta b\|_{\dot{H}^{2m+1}}^2\md\tau\big)^{\frac{1}{2}}
\big(\int_0^t\|\nabla u\|_{H^2}^2\md\tau\big)^{\frac{1}{2}}\\
&+\sup_{0\leq\tau\leq t}\|b\|_{L^2}^{\frac{2}{3}}\|b\|_{H^{2s+4}}^{\frac{1}{3}}
\big(\int_0^t\|\partial_\theta b\|_{\dot{H}^{2m+1}}^2\md\tau\big)^{\frac{1}{3}}
\big(\int_0^t\|\partial_\theta b\|_{\dot{H}^3}^2\md\tau\big)^{\frac{1}{6}}
\big(\int_0^t\|\nabla u\|_{H^{2s+2}}^2\md\tau\big)^{\frac{1}{2}}\\
&+\sup_{0\leq\tau\leq t}
\|b\|_{H^{2s+2}}
\big(\int_0^t\|\nabla u\|_{H^{2s+1}}^2\md\tau\big)^{\frac{1}{2}}
\big(\int_0^t\|\nabla \partial_\theta b\|_{H^2}^2\md\tau\big)^{\frac{1}{2}}\\
\lesssim& E_0(t)e_0(t)^{\frac{1}{2}}+E_0(t)^{\frac{1}{2}}e_0(t).
\end{align*}

Integrating (\ref{3.5}) in time and using the above estimates gives
\begin{align}
\begin{aligned}
&\sup_{0\leq\tau\leq t}
\big(\|u_t,b_t\|_{\dot{H}^{2m}}^2+\|\partial_\theta u,\partial_\theta b\|_{\dot{H}^{2m}}^2
+\|u,b\|_{\dot{H}^{2m+2}}^2\big)\\
&+\int_0^t\big(\|u_\tau,b_\tau\|_{\dot{H}^{2m+1}}^2
+\|\partial_\theta u,\partial_\theta b\|_{\dot{H}^{2m+1}}^2\big)\md\tau\\
\lesssim&\varepsilon^2+e_0(t)^{\frac{3}{2}}
+E_0(t)^{\frac{1}{2}}e_0(t)
+E_0(t)e_0(t)^{\frac{1}{2}}
+E_0(t)^{\frac{3}{4}}e_0(t)^{\frac{3}{4}}.
\end{aligned}
\end{align}
So far, we have successfully proved this lemma.
%Using Young's and Gagliardo-Nirenberg inequalities, We further obtain
%\begin{align}\label{3.5*}
%\begin{aligned}
%&\sup_{0\leq\tau\leq t}\{\|u_\tau,b_\tau\|_{2m}^2+A\|u_\tau,b_\tau\|_0^2+\|u,b\|_{2m+1}^2
%+\|\partial_\theta u,\partial_\theta b\|_{2m}^2+A\|\partial_\theta u,\partial_\theta b\|_0^2\}\\
%&+\int_0^t\{\|u_\tau,b_\tau\|_{2m+1}^2+A\|u_\tau,b_\tau\|_1^2+\|\partial_\theta u,\partial_\theta b\|_{2m}^2\}\md\tau\big\}\\
%\lesssim&\varepsilon^2+\int_0^tN\md\tau
%\end{aligned}
%\end{align}
%Where we have calculated
%\begin{align*}
%\|f\|_{m}^2\leq\frac{1}{4}\|f\|_{m+1}^2+\|f\|_0^2\\
%\|f\|_{m+1}^2\leq\frac{1}{4}\|f\|_{m+2}^2+\|f\|_1^2.
%\end{align*}
\end{proof}

Finally, we establish the decay estimates.
\begin{lem}\label{lem3.4}
From the definition in (\ref{2.4}), for $1\leq m\leq s-1$, $s\geq 2$, and $\frac{3}{23}\leq\sigma<1$, we have
\begin{align}\label{4}
e_1(t)\lesssim\varepsilon^2+e_0(t)
+e_0(t)^{\frac{1}{2}}e_1(t)
+E_0(t)^{\frac{13}{72}}
E_1(t)^{\frac{1}{3+\sigma}}
e_1(t)^{\frac{1+\sigma}{2(3+\sigma)}+\frac{59}{72}}.
\end{align}
\end{lem}

\begin{proof}
This proof is similar to Lemma \ref{lem3.3}.
The difference is that after multiplying the time-weight $(1+t)^2$, extra linear terms will appear on the right-hand side of (\ref{3.6}), in comparison to (\ref{3.5}).
By utilizing Gagliardo-Nirenberg inequalities and Young's inequality, we can easily control these terms.
Following the same steps as deriving the inequality (\ref{3.5}), we obtain
\begin{align}\label{3.6}
\begin{aligned}
&\frac{\md}{\md t}(1+t)^2
\big(\|u_t,b_t\|_{\dot{H}^{2m}}^2+\|\partial_\theta u,\partial_\theta b\|_{\dot{H}^{2m}}^2
+\|u,b\|_{\dot{H}^{2m+2}}^2\big)\\
&+(1+t)^2\big(\|u_t, b_t\|_{\dot{H}^{2m+1}}^2
+\|\partial_\theta u,\partial_\theta b\|_{\dot{H}^{2m+1}}^2\big)\\
\lesssim&(1+t)\big(\|u_t,b_t\|_{\dot{H}^{2m}}^2
+\|\partial_\theta u,\partial_\theta b\|_{\dot{H}^{2m}}^2
+\|u,b\|_{\dot{H}^{2m+2}}^2\big)
+(1+t)^2(N_1+N_2).
\end{aligned}
\end{align}
We will begin with estimating the linear terms resulting from the time-weight.
Applying the interpolation inequality and Young's inequality, it can be inferred that
\begin{align*}
&\int_0^t(1+\tau)\|u_\tau,b_\tau\|_{\dot{H}^{2m}}^2\md\tau\\
\leq &C\int_0^t(1+\tau)\|u_\tau,b_\tau\|_{\dot{H}^{2m-1}}\|u_\tau,b_\tau\|_{\dot{H}^{2m+1}}\md\tau\\
\leq&\frac{1}{8}\int_0^t(1+\tau)^2\|u_\tau,b_\tau\|_{\dot{H}^{2m+1}}^2\md\tau
+C\int_0^t\|u_\tau,b_\tau\|_{\dot{H}^{2m-1}}^2\md\tau\\
\leq&\frac{1}{8}\int_0^t(1+\tau)^2\|u_\tau,b_\tau\|_{\dot{H}^{2m+1}}^2\md\tau+Ce_0(t),
\end{align*}
provided $m\geq1$.
Similarly, the second linear term can be bound as follows:
\begin{align*}
&\int_0^t(1+\tau)\|\partial_\theta u,\partial_\theta b\|_{\dot{H}^{2m}}^2\md\tau\\
\leq&\frac{1}{8}\int_0^t(1+\tau)^2\|\partial_\theta u,\partial_\theta b\|_{\dot{H}^{2m+1}}^2\md\tau+Ce_0(t).
\end{align*}
For the third linear term, using the interpolation inequality, Young's inequality, Lemma \ref{lem2.1} and (\ref{007}), one has
\begin{align*}
&\int_0^t(1+\tau)\|u,b\|_{\dot{H}^{2m+2}}^2\md\tau\\
\leq&C\int_0^t(1+\tau)\|u,b\|_{\dot{H}^{2m+1}}\|u,b\|_{\dot{H}^{2m+3}}\md\tau\\
\leq&\frac{1}{8}\int_0^t(1+\tau)^2\|u,b\|_{\dot{H}^{2m+1}}^2\md\tau
+C\int_0^t\|u,b\|_{\dot{H}^{2m+3}}^2\md\tau\\
\leq&\frac{1}{8}\int_0^t(1+\tau)^2\|\partial_\theta u,\partial_\theta b\|_{\dot{H}^{2m+1}}^2\md\tau+Ce_0(t),
\end{align*}
provided $m\leq s-1$.
%We now turn to the last two terms. By H\"{o}lder's inequality, Young's inequality and (\ref{007}), we deduce
%\begin{align*}
%&\int_0^t(1+\tau)\langle\nabla^{2m+1}b, \nabla^{2m+1}b_\tau\rangle\md\tau\\
%\leq&C\int_0^t(1+\tau)\|b_\tau\|_{\dot{H}^{2m+1}}\|b\|_{\dot{H}^{2m+1}}\md\tau\\
%\leq&\frac{1}{8}\int_0^t(1+\tau)^2\|b_\tau\|_{\dot{H}^{2m+1}}^2\md\tau
%+C\int_0^t\|\partial_\theta b\|_{\dot{H}^{2m+1}}^2\md\tau\\
%\leq&\frac{1}{8}\int_0^t(1+\tau)^2\|b_\tau\|_{\dot{H}^{2m+1}}^2\md\tau+Ce_0(t).
%\end{align*}
%From Lemma \ref{lem2.1}, the same estimate holds for
%\begin{align*}
%&\int_0^t(1+\tau)\langle\nabla^{2m+1}u, \nabla^{2m+1}u_t\rangle\md\tau\\
%\leq&\frac{1}{8}\int_0^t(1+\tau)^2\|u_t\|_{\dot{H}^{2m+1}}^2\md\tau+Ce_0(t).
%\end{align*}

We now turn to the nonlinear terms. The difference from the proof in Lemma \ref{lem3.3} is that in this instance, a reasonable distribution of the time-weight $(1+t)^2$ also needs to be considered.
For $(1+t)^2N_1$, we begin by examining the lower-order derivative terms:
\begin{align*}
&(1+t)^2\langle\Delta^{2m}u_t, F_t-\partial_{\theta} G\rangle+\langle\Delta^{2m}b_t, G_t-\partial_{\theta} F\rangle\\
=&(1+t)^2\langle\nabla^{2m}u_t, \nabla^{2m}(b\cdot\nabla b)_t\rangle
-(1+t)^2\langle\nabla^{2m}u_t,\nabla^{2m}(u\cdot\nabla u)_t\rangle\\
&-(1+t)^2\langle\nabla^{2m}u_t, \nabla^{2m}\partial_{\theta}(b\cdot\nabla u)\rangle
+(1+t)^2\langle\nabla^{2m}u_t,\nabla^{2m}\partial_{\theta}(u\cdot\nabla b)\rangle\\
&+(1+t)^2\langle\nabla^{2m}b_t, \nabla^{2m}(b\cdot\nabla u)_t\rangle
-(1+t)^2\langle\nabla^{2m}b_t,\nabla^{2m}(u\cdot\nabla b)_t\rangle\\
&-(1+t)^2\langle\nabla^{2m}b_t, \nabla^{2m}\partial_{\theta}(b\cdot\nabla b)\rangle
+(1+t)^2\langle\nabla^{2m}b_t,\nabla^{2m}\partial_{\theta}(u\cdot\nabla u)\rangle.
\end{align*}
Like (\ref{3.7}) in Lemma \ref{lem3.3}, our focus is to estimate the first term. After some adjustments, the remaining terms can be handled through the similar method. We will omit details here.
Using integrating by parts, Lemma \ref{lem2.2}, Sobolev imbedding theorem and (\ref{007}), for $m\geq 1$, one may deduce that
\begin{align*}
&(1+t)^2\langle\nabla^{2m}u_t, \nabla^{2m}(b\cdot\nabla b)_t\rangle\\
=&-(1+t)^2\langle\nabla^{2m+1}u_t, \nabla^{2m}(b\otimes b)_t\rangle\\
\lesssim&(1+t)^2\|u_t\|_{\dot{H}^{2m+1}}\big(\|b\|_{\dot{H}^{2m}}\|b_t\|_{L^\infty}
+\|b\|_{L^\infty}\|b_t\|_{\dot{H}^{2m}}\big)\\
\lesssim&(1+t)^2\|u_t\|_{\dot{H}^{2m+1}}\big(\|\partial_\theta b\|_{\dot{H}^{2m-2}}^{\frac{1}{3}}
\|\partial_\theta b\|_{\dot{H}^{2m+1}}^{\frac{2}{3}}
\|b_t\|_{L^2}^{\frac{2}{3}}\|b_t\|_{\dot{H}^3}^{\frac{1}{3}}\\
&+\|\partial_\theta b\|_{L^2}^{\frac{2}{3}}\|\partial_\theta b\|_{\dot{H}^3}^{\frac{1}{3}}
\|b_t\|_{\dot{H}^{2m-2}}^{\frac{1}{3}}
\|b_t\|_{\dot{H}^{2m+1}}^{\frac{2}{3}}\big).
\end{align*}
Hence
\begin{align*}
&\int_0^t(1+\tau)^2\langle\nabla^{2m}u_\tau, \nabla^{2m}(b\cdot\nabla b)_\tau\rangle\md\tau\\
\lesssim&\sup_{0\leq\tau\leq t}\|\partial_\theta b\|_{\dot{H}^{2m-2}}^{\frac{1}{3}}
\|b_\tau\|_{L^2}^{\frac{2}{3}}
\big(\int_0^t(1+\tau)^2\|u_\tau\|_{\dot{H}^{2m+1}}^2\md\tau\big)^{\frac{1}{2}}\\
&\times\big(\int_0^t(1+\tau)^2\|b_\tau\|_{\dot{H}^3}^2\md\tau\big)^{\frac{1}{6}}
\big(\int_0^t(1+\tau)^2\|\partial_\theta b\|_{\dot{H}^{2m+1}}^2\md\tau\big)^{\frac{1}{3}}\\
&+\sup_{0\leq\tau\leq t}\|b_\tau\|_{\dot{H}^{2m-2}}^{\frac{1}{3}}
\|\partial_\theta b\|_{L^2}^{\frac{2}{3}}
\big(\int_0^t(1+\tau)^2\|u_\tau\|_{\dot{H}^{2m+1}}^2\md\tau\big)^{\frac{1}{2}}\\
&\times\big(\int_0^t(1+\tau)^2\|b_\tau\|_{\dot{H}^{2m+1}}^2\md\tau\big)^{\frac{1}{3}}
\big(\int_0^t(1+\tau)^2\|\partial_\theta b\|_{\dot{H}^3}^2\md\tau\big)^{\frac{1}{6}}\\
\lesssim& e_0(t)^{\frac{1}{2}}e_1(t).
\end{align*}

For the higher-order term of $(1+t)^2N_1$, from incompressible constraints, Lemma \ref{lem2.2}, Lemma \ref{lem2.1} and (\ref{007}), one has
\begin{align*}
&(1+t)^2\langle\Delta^{2m}b_t, -\Delta G\rangle\\
=&(1+t)^2\langle\nabla^{2m+1}b_t, \nabla^{2m+2}(u\otimes b-b\otimes u)\rangle\\
\lesssim&(1+t)^2\|b_t\|_{\dot{H}^{2m+1}}\big(\|u\|_{L^\infty}\|b\|_{\dot{H}^{2m+2}}
+\|b\|_{L^\infty}\|u\|_{\dot{H}^{2m+2}}\big)\\
\lesssim&(1+t)^2\|b_t\|_{\dot{H}^{2m+1}}
\big(\|\partial_\theta u\|_{L^2}^{\frac{2}{3}}
\|\partial_\theta u\|_{\dot{H}^3}^{\frac{1}{3}}
\|\partial_\theta b\|_{\dot{H}^{2m+1}}^{\frac{2}{3}}
\|b\|_{\dot{H}^{2m+4}}^{\frac{1}{3}}\\
&+\|\partial_\theta b\|_{L^2}^\frac{2}{3}
\|\partial_\theta b\|_{\dot{H}^3}^\frac{1}{3}
\|\partial_\theta u\|_{\dot{H}^{2m+1}}^\frac{2}{3}
\|u\|_{\dot{H}^{2m+4}}^\frac{1}{3}\big).
\end{align*}
Therefore
\begin{align*}
&\int_0^t(1+\tau)^2\langle\Delta^{2m}b_\tau, -\Delta G\rangle\md\tau\\
\lesssim&\sup_{0\leq\tau\leq t}
\|\partial_\theta u\|_{L^2}^{\frac{2}{3}}\|b\|_{\dot{H}^{2m+4}}^{\frac{1}{3}}
\big(\int_0^t(1+\tau)^2\|b_\tau\|_{\dot{H}^{2m+1}}^2\md\tau\big)^{\frac{1}{2}}\\
&\times\big(\int_0^t(1+\tau)^2\|\partial_\theta u\|_{\dot{H}^3}^2\md\tau\big)^{\frac{1}{6}}
\big(\int_0^t(1+\tau)^2\|\partial_\theta b\|_{\dot{H}^{2m+1}}^2\md\tau\big)^{\frac{1}{3}}\\
&+\sup_{0\leq\tau\leq t}
\|\partial_\theta b\|_{L^2}^\frac{2}{3}\|u\|_{\dot{H}^{2m+4}}^\frac{1}{3}
\big(\int_0^t(1+\tau)^2\|b_\tau\|_{\dot{H}^{2m+1}}^2\md\tau\big)^{\frac{1}{2}}\\
&\times\big(\int_0^t(1+\tau)^2\|\partial_\theta b\|_{\dot{H}^3}^2\md\tau\big)^{\frac{1}{6}}
\big(\int_0^t(1+\tau)^2\|\partial_\theta u\|_{\dot{H}^{2m+1}}^2\md\tau\big)^{\frac{1}{3}}\\
\lesssim& e_0(t)^{\frac{1}{2}}e_1(t).
\end{align*}

The same idea holds for $(1+t)^2N_2$. We still put the same order derivative terms together
\begin{align*}
\begin{aligned}
&(1+t)^2\langle\Delta^{2m+1}u, F_t-\partial_{\theta} G\rangle+\langle\Delta^{2m+1}b, G_t-\partial_{\theta} F\rangle\\
=&-(1+t)^2\langle\nabla^{2m+1}u, \nabla^{2m+1}(b\cdot\nabla b)_t\rangle
+(1+t)^2\langle\nabla^{2m+1}u,\nabla^{2m+1}(u\cdot\nabla u)_t\rangle\\
&+(1+t)^2\langle\nabla^{2m+1}u, \nabla^{2m+1}\partial_{\theta}(b\cdot\nabla u)\rangle
-(1+t)^2\langle\nabla^{2m+1}u,\nabla^{2m+1}\partial_{\theta}(u\cdot\nabla b)\rangle\\
&-(1+t)^2\langle\nabla^{2m+1}b, \nabla^{2m+1}(b\cdot\nabla u)_t\rangle
+(1+t)^2\langle\nabla^{2m+1}b,\nabla^{2m+1}(u\cdot\nabla b)_t\rangle\\
&+(1+t)^2\langle\nabla^{2m+1}b, \nabla^{2m+1}\partial_{\theta}(b\cdot\nabla b)\rangle
-(1+t)^2\langle\nabla^{2m+1}b,\nabla^{2m+1}\partial_{\theta}(u\cdot\nabla u)\rangle.
\end{aligned}
\end{align*}
According to the same linear structure of $u$ and $b$ described in (\ref{3.1}) and the same regularity
in (\ref{3.6}), we only need to estimate one of the above terms.
Without loss of generality, we focus on estimating the first term, while the remaining terms can be dealt with similarly. By $\nabla\cdot b=0$, integrating by parts, Lemma \ref{lem2.2} and (\ref{007}), we have
\begin{align*}
&-(1+t)^2\langle\nabla^{2m+1}u, \nabla^{2m+1}(b\cdot\nabla b)_t\rangle\\
=&(1+t)^2\langle\nabla^{2m+2}u, \nabla^{2m+1}(b\otimes b)_t\rangle\\
\lesssim&(1+t)^2\|u\|_{\dot{H}^{2m+2}}\big(\|b\|_{\dot{H}^{2m+1}}\|b_t\|_{L^\infty}
+\|b\|_{L^\infty}\|b_t\|_{\dot{H}^{2m+1}}\big)\\
\lesssim&(1+t)^2\|\partial_\theta u\|_{\dot{H}^{2m+1}}^{\frac{2}{3}}\|u\|_{\dot{H}^{2m+4}}^{\frac{1}{3}}
\big(\|\partial_\theta b\|_{\dot{H}^{2m+1}}\|b_t\|_{L^2}^{\frac{2}{3}}\|b_t\|_{\dot{H}^3}^{\frac{1}{3}}
+\|\partial_\theta b\|_{L^2}^{\frac{2}{3}}\|\partial_\theta b\|_{\dot{H}^3}^{\frac{1}{3}}
\|b_t\|_{\dot{H}^{2m+1}}\big).
\end{align*}
Hence
\begin{align*}
&\int_0^t-(1+\tau)^2\langle\nabla^{2m+1}u, \nabla^{2m+1}(b\cdot\nabla b)_\tau\rangle\md\tau\\
\lesssim&\sup_{0\leq\tau\leq t}\|b_\tau\|_{L^2}^{\frac{2}{3}}\|u\|_{\dot{H}^{2m+4}}^{\frac{1}{3}}
\big(\int_0^t(1+\tau)^2\|\partial_\theta u\|_{\dot{H}^{2m+1}}^2\md\tau\big)^{\frac{1}{3}}\\
&\times\big(\int_0^t(1+\tau)^2\|\partial_\theta b\|_{\dot{H}^{2m+1}}^2\md\tau\big)^{\frac{1}{2}}
\big(\int_0^t(1+\tau)^2\|b_\tau\|_{\dot{H}^3}^2\md\tau\big)^{\frac{1}{6}}\\
&+\sup_{0\leq\tau\leq t}\|\partial_\theta b\|_{L^2}^{\frac{2}{3}}
\|u\|_{\dot{H}^{2m+4}}^{\frac{1}{3}}
\big(\int_0^t(1+\tau)^2\|\partial_\theta u\|_{\dot{H}^{2m+1}}^2\md\tau\big)^{\frac{1}{3}}\\
&\times\big(\int_0^t(1+\tau)^2\|\partial_\theta b\|_{\dot{H}^3}^2\md\tau\big)^{\frac{1}{6}}
\big(\int_0^t(1+\tau)^2\|b_\tau\|_{\dot{H}^{2m+1}}^2\md\tau\big)^{\frac{1}{2}}\\
\lesssim& e_0(t)^{\frac{1}{2}}e_1(t).
\end{align*}

For the remaining term of $(1+t)^2N_2$, using $\nabla\cdot u=0$, and Lemma \ref{lem2.2}, we have
\begin{align}\label{mild}
\begin{aligned}
&(1+t)^2\langle\Delta^{2m+1}b, \Delta G\rangle\\
=&(1+t)^2\langle\nabla^{2m+2}b, \nabla^{2m+2}(b\cdot\nabla u-u\cdot\nabla b)\rangle\\
=&(1+t)^2\langle\nabla^{2m+2}b, \nabla^{2m+2}(b\cdot\nabla u)\rangle
-(1+t)^2\langle\nabla^{2m+2}b, [\nabla^{2m+2},u\cdot\nabla]b\rangle\\
\lesssim&(1+t)^2\|b\|_{\dot{H}^{2m+2}}\big(\|b\|_{\dot{H}^{2m+2}}\|\nabla u\|_{L^\infty}
+\|u\|_{\dot{H}^{2m+2}}\|\nabla b\|_{L^\infty}
+\|b\|_{L^\infty}\|\nabla u\|_{\dot{H}^{2m+2}}\big).
\end{aligned}
\end{align}
For the first two terms in the last line of the above equation,
using interpolation inequalities, Lemma \ref{lem2.1} and (\ref{007}), we obtain
\begin{align*}
&\int_0^t(1+\tau)^2\|b\|_{\dot{H}^{2m+2}}\big(\|b\|_{\dot{H}^{2m+2}}\|\nabla u\|_{L^\infty}
+\|u\|_{\dot{H}^{2m+2}}\|\nabla b\|_{L^\infty}\big)\md\tau\\
\lesssim&\int_0^t(1+\tau)^2\|b\|_{\dot{H}^{2m+2}}
\big(\|\partial_\theta b\|_{\dot{H}^{2m+1}}^{\frac{1}{2}}
\|\partial_\theta b\|_{\dot{H}^{2m+3}}^{\frac{1}{2}}
\|\partial_\theta u\|_{\dot{H}^1}^{\frac{1}{2}}
\|\partial_\theta u\|_{\dot{H}^3}^{\frac{1}{2}}\\
&+\|\partial_\theta u\|_{\dot{H}^{2m+1}}^{\frac{1}{2}}
\|\partial_\theta u\|_{\dot{H}^{2m+3}}^{\frac{1}{2}}
\|\partial_\theta b\|_{\dot{H}^1}^{\frac{1}{2}}
\|\partial_\theta b\|_{\dot{H}^3}^{\frac{1}{2}}\big)\md\tau\\
\lesssim&\sup_{0\leq\tau\leq t}(1+\tau)\|b\|_{\dot{H}^{2m+2}}
\big(\int_0^t(1+\tau)^2\|\partial_\theta b\|_{\dot{H}^{2m+1}}^2\md\tau\big)^{\frac{1}{4}}
\big(\int_0^t\|\partial_\theta b\|_{\dot{H}^{2m+3}}^2\md\tau\big)^{\frac{1}{4}}\\
&\times\big(\int_0^t\|\partial_\theta u\|_{\dot{H}^1}^2\md\tau\big)^{\frac{1}{4}}
\big(\int_0^t(1+\tau)^2\|\partial_\theta u\|_{\dot{H}^3}^2\md\tau\big)^{\frac{1}{4}}\\
&+\sup_{0\leq\tau\leq t}(1+\tau)\|b\|_{\dot{H}^{2m+2}}
\big(\int_0^t(1+\tau)^2\|\partial_\theta u\|_{\dot{H}^{2m+1}}^2\md\tau\big)^{\frac{1}{4}}
\big(\int_0^t\|\partial_\theta u\|_{\dot{H}^{2m+3}}^2\md\tau\big)^{\frac{1}{4}}\\
&\times\big(\int_0^t\|\partial_\theta b\|_{\dot{H}^1}^2\md\tau\big)^{\frac{1}{4}}
\big(\int_0^t(1+\tau)^2\|\partial_\theta b\|_{\dot{H}^3}^2\md\tau\big)^{\frac{1}{4}}\\
\lesssim&e_0(t)^{\frac{1}{2}}e_1(t).
\end{align*}
For the last term in the last line of (\ref{mild}), the most wild term, we shall carefully balance the indicators between the time-weight $(1+t)^2$ and interpolation inequalities, so that the term can be controlled by the energies defined in Section 2. Using interpolation inequalities,
Lemma \ref{lem2.1} and (\ref{007}), we can obtain
\begin{align*}
&(1+t)^2\|b\|_{\dot{H}^{2m+2}}\|b\|_{L^\infty}\|\nabla u\|_{\dot{H}^{2m+2}}\\
\lesssim&(1+t)^2\|\partial_\theta b\|_{\dot{H}^{2m+2}}^{\frac{2}{9}}
\|\partial_\theta b\|_{\dot{H}^{2m+1}}^{\frac{2}{3}}
\|b\|_{\dot{H}^{2m+8}}^{\frac{1}{9}}
\|b\|_{\dot{H}^{-\sigma}}^{\frac{2}{3+\sigma}}
\|b\|_{\dot{H}^3}^{\frac{1+\sigma}{3+\sigma}}
\|\partial_\theta u\|_{\dot{H}^{2m+1}}^{\frac{3}{4}}
\|u\|_{\dot{H}^{2m+9}}^{\frac{1}{4}}\\
=&C(1+t)^{\frac{2}{9}}\|\partial_\theta b\|_{\dot{H}^{2m+2}}^{\frac{2}{9}}
(1+t)^{\frac{2}{3}}\|\partial_\theta b\|_{\dot{H}^{2m+1}}^{\frac{2}{3}}
\|b\|_{\dot{H}^{2m+8}}^{\frac{1}{9}}
\|b\|_{\dot{H}^{-\sigma}}^{\frac{2}{3+\sigma}}
(1+t)^{\frac{1+\sigma}{3+\sigma}-\frac{1}{3}}
\|b\|_{\dot{H}^3}^{\frac{1+\sigma}{3+\sigma}-\frac{1}{3}}\\
&\times(1+t)^{\frac{1}{3}}\|\partial_\theta b\|_{\dot{H}^3}^{\frac{1}{3}}
(1+t)^{\frac{3}{4}}\|\partial_\theta u\|_{\dot{H}^{2m+1}}^{\frac{3}{4}}
\|\nabla u\|_{\dot{H}^{2m+8}}^{\frac{1}{4}}.
\end{align*}
Note that
\begin{align*}
\|b\|_{\dot{H}^3}^{\frac{1+\sigma}{3+\sigma}-\frac{1}{3}}
\lesssim\|\partial_\theta b\|_{\dot{H}^2}^{\frac{1+\sigma}{2(3+\sigma)}-\frac{1}{6}}
\|b\|_{\dot{H}^4}^{\frac{1+\sigma}{2(3+\sigma)}-\frac{1}{6}},
\end{align*}
therefore
\begin{align*}
&\int_0^t(1+\tau)^2\|b\|_{\dot{H}^{2m+2}}\|b\|_{L^\infty}
\|\nabla u\|_{\dot{H}^{2m+2}}\md\tau\\
\lesssim&\sup_{0\leq\tau\leq t}
(1+\tau)^{\frac{2}{9}}\|\partial_\theta b\|_{\dot{H}^{2m+2}}^{\frac{2}{9}}
\|b\|_{H^{2s+6}}^{\frac{1}{9}}
\|b\|_{\dot{H}^{-\sigma}}^{\frac{2}{3+\sigma}}
(1+\tau)^{\frac{1+\sigma}{2(3+\sigma)}-\frac{1}{6}}
\|\partial_\theta b\|_{\dot{H}^2}^{\frac{1+\sigma}{2(3+\sigma)}-\frac{1}{6}}\\
&\times
(1+\tau)^{\frac{1+\sigma}{2(3+\sigma)}-\frac{1}{6}}
\|b\|_{\dot{H}^4}^{\frac{1+\sigma}{2(3+\sigma)}-\frac{1}{6}}
\big(\int_0^t(1+\tau)^2\|\partial_\theta b\|_{\dot{H}^3}^2\md\tau\big)^{\frac{1}{6}}
\big(\int_0^t(1+\tau)^2\|\partial_\theta b\|_{\dot{H}^{2m+1}}^2\md\tau\big)^{\frac{1}{3}}\\
&\times
\big(\int_0^t(1+\tau)^2\|\partial_\theta u\|_{\dot{H}^{2m+1}}^2\md\tau\big)^{\frac{3}{8}}
\big(\int_0^t\|\nabla u\|_{H^{2s+6}}^2\md\tau\big)^{\frac{1}{8}}\\
\lesssim&E_0(t)^{\frac{13}{72}}
E_1(t)^{\frac{1}{3+\sigma}}
e_1(t)^{\frac{1+\sigma}{2(3+\sigma)}+\frac{59}{72}},
\end{align*}
provided $\frac{3}{23}\leq\sigma<1$.

Integrating (\ref{3.6}) respect to time and invoking the above estimates yields
\begin{align*}
\begin{aligned}
&\sup_{0\leq\tau\leq t}(1+\tau)^2
\big(\|u_t,b_t\|_{\dot{H}^{2m}}^2+\|\partial_\theta u,\partial_\theta b\|_{\dot{H}^{2m}}^2
+\|u,b\|_{\dot{H}^{2m+2}}^2\big)\\
&+\int_0^t(1+\tau)^2\big(\|u_\tau, b_\tau\|_{\dot{H}^{2m+1}}^2
+\|\partial_\theta u,\partial_\theta b\|_{\dot{H}^{2m+1}}^2\big)\md\tau\\
\lesssim&\varepsilon^2+e_0(t)
+e_0(t)^{\frac{1}{2}}e_1(t)
+E_0(t)^{\frac{13}{72}}
E_1(t)^{\frac{1}{3+\sigma}}
e_1(t)^{\frac{1+\sigma}{2(3+\sigma)}+\frac{59}{72}}.
\end{aligned}
\end{align*}
This finishes the proof of this lemma.
\end{proof}

\section{Proof of Theorem \ref{thm0.1}}

With all the energy estimates provided above, we are now ready to complete the proof of Theorem \ref{thm0.1}. To do this, we first define the total energy:
\begin{align*}
E(t)=E_0(t)+E_1(t)+e_0(t)+e_1(t).
\end{align*}
Multiplying $(\ref{1})$, $(\ref{2})$, $(\ref{3})$ and $(\ref{4})$ by different suitable numbers and combining them, we obtain the following \emph{a priori} estimate
\begin{align}\label{5}
E(t)\leq CE(0)+CE(t)^{\frac{3}{2}},
\end{align}
for some positive constant $C$.

Once we have (\ref{5}), along with the initial condition (\ref{**}) and the application of the bootstrapping argument, we can conclude that
\begin{align*}
E(t)\lesssim E(0),
\end{align*}
for all $t>0$.

The establishment of unique local smooth solutions follows a standard procedure (see the book \cite{T2001, MB2002} for further details). We then complete the proof of Theorem \ref{thm0.1}.

\vspace{.2in}%%%%%

\textbf{Acknowledgements}
The research of Yi Zhou
was supported by the National Natural Science Foundation
of China (No. 12171097), Shanghai Science and Technology Program (No. 21JC1400600)
and Shanghai Key Laboratory
for Contemporary Applied Mathematics, Fudan
University.
The authors sincerely thank Prof. Yi Zhu for helpful discussions.

\end{document}